\documentclass[12pt,reqno]{amsart}

\usepackage[colorlinks,hyperindex]{hyperref}
	\hypersetup
	{
		colorlinks,%
 		citecolor=blue,%
		linkcolor=blue,%
		urlcolor=blue,%
	}

\usepackage{amssymb}

\newtheorem{theorem}{Theorem}[section]
\newtheorem{lemma}[theorem]{Lemma}
\newtheorem{proposition}[theorem]{Proposition}
\newtheorem{corollary}[theorem]{Corollary}
	\newcommand{\dd}{\delta}

\theoremstyle{remark}

\newtheorem{definition}[theorem]{Definition}

\numberwithin{equation}{section}

\begin{document}

\baselineskip=15.5pt

\title[Brauer group of moduli of Higgs bundles and connections]{Brauer
group of moduli of Higgs bundles and connections}

\author[D. Baraglia]{David Baraglia}

\address{School of Mathematical Sciences, The University of Adelaide, 
Adelaide SA 5005, Australia}

\email{david.baraglia@adelaide.edu.au}

\author[I. Biswas]{Indranil Biswas}

\address{School of Mathematics, Tata Institute of Fundamental
Research, Homi Bhabha Road, Bombay 400005, India}

\email{indranil@math.tifr.res.in}

\author[L. P. Schaposnik]{Laura P. Schaposnik}

\address{Department of Mathematics, University of Illinois, Chicago, IL 60607, USA}

\email{schapos@illinois.edu}



\date{}

\begin{abstract}
Given a compact Riemann surface $X$ and a semisimple affine algebraic group $G$
defined over $\mathbb C$, there are moduli spaces of Higgs bundles and of
connections associated to $(X,\, G)$. We compute the Brauer group of the
smooth locus of these varieties.
\end{abstract}

\maketitle

\section{Introduction}\label{sec1}

We dedicate this paper to the study of the Brauer group of the moduli spaces of certain Higgs bundles and of holomorphic connections on a Riemann surface. 
Recall that given a complex quasiprojective variety $Z$, its Brauer group $\text{Br}(Z)$ is the Morita equivalence classes of Azumaya algebras over $Z$. This group
coincides with the equivalence classes of principal $\text{PGL}$--bundles over $Z$,
where two principal $\text{PGL}$--bundles
$P$ and $Q$ are identified if there are vector bundles $V$ and $W$ over
$Z$ such that the two principal $\text{PGL}$--bundles $P\otimes {\mathbb P}(V)$ and
$Q\otimes {\mathbb P}(W)$ are isomorphic. The cohomological Brauer group
$\text{Br}'(Z)$ of the variety $Z$ is the torsion part of the \'etale cohomology group
$H^2(Z,\, {\mathbb G}_m)$. There is a natural injective homomorphism
$\text{Br}(Z)\,\longrightarrow\, \text{Br}'(Z)$ which is in fact an isomorphism
by a theorem of Gabber \cite{dJ}, \cite{Ho}.

Consider now a compact connected Riemann surface $X$ of genus $g\geq 3$. Given a fixed
base point $x_0$ and two integers $r\, \geq\,2$ and
$\dd$, we let 
${\mathcal M}_C$ denote the moduli space of all logarithmic connections
$(E,\, D)$ on $X$, singular over $x_0$, satisfying the following four natural conditions:
\begin{enumerate}
\item[I.] $E$ is a holomorphic vector bundle on $X$ of rank $r$ with
$\bigwedge\nolimits^r E\,=\, {\mathcal O}_X(\dd x_0)$,

\item[II.] the logarithmic connection on $\bigwedge\nolimits^r E\,=\,
{\mathcal O}_X(\dd x_0)$
induced by $D$ coincides with the connection on ${\mathcal O}_X(\dd x_0)$ defined
by the de Rham differential,

\item[III.] the residue of $D$ at $x_0$ is $-\frac{\dd}{r}\text{Id}_{E_{x_0}}$, and

\item[IV.] there is no holomorphic subbundle $F\, \subset\, E$ with
$1\,\leq\,\text{rank}(F)\, <\, r$ such that $D$ preserves $F$. 
\end{enumerate}
This moduli space ${\mathcal M}_C$ has a natural projective bundle
once we fix a point of $X$,
$${\mathbb P}_C\, \longrightarrow\, {\mathcal M}_C\, ,$$
through which in Section \ref{hola} we study the Brauer group ${\rm Br}({\mathcal M}_C)$: 

\noindent {\bf Theorem \ref{theo1}.}
\textit{The Brauer group ${\rm Br}({\mathcal M}_C)$ is isomorphic to ${\mathbb 
Z}/\tau\mathbb Z$, where $\tau\,=\, {\rm g.c.d.}(r,\dd)$. The group ${\rm 
Br}({\mathcal M}_C)$ is generated by the class of ${\mathbb P}_C$.}

Fixing the compact connected Riemann surface $X$ and the invariant $\dd$, one can 
also compute the analytic Brauer group of the $\text{SL}(r,{\mathbb C})$--character 
variety ${\mathcal R}$ associated to the pair $(X,\, \dd)$.

\noindent{\bf Theorem \ref{thm1}.} {\it The analytic cohomological Brauer group ${\rm Br}'_{\rm an}({\mathcal R})$ is
isomorphic to a quotient of the cyclic group ${\mathbb Z}/\tau
\mathbb Z$, where $\tau\,=\, {\rm g.c.d.}(r,\dd)$. The group
${\rm Br}_{\rm an}({\mathcal R})$ is generated by
the class of a naturally associated projective bundle $\mathbb{P}_R$.}

Over the compact connected Riemann surface $X$ one may also consider the moduli space 
 ${\mathcal M}_H$ of stable
Higgs bundles on $X$ of the form $(E,\, \Phi)$, where
\begin{itemize}\item $E$ is a holomorphic vector bundle
of rank $r$ with $\bigwedge\nolimits^r E\,=\, {\mathcal O}_X(\dd x_0)$, and

\item $\Phi$ is a
Higgs field on $X$ with $\text{trace}(\Phi)\,=\, 0$.
\end{itemize}
The moduli space ${\mathcal M}_H$ is a
smooth quasiprojective variety which also has a natural projective bundle
$${\mathbb P}_H\, \longrightarrow\, {\mathcal M}_H$$
once we fix a point of $X$.
In Section \ref{sec2.3} we study the Brauer group of ${\mathcal M}_H$ and prove the following:

\noindent{\bf Proposition \ref{mas}.} {\it The group ${\rm Br}({\mathcal M}_H)$ is isomorphic to the cyclic
group ${\mathbb Z}/\tau\mathbb Z$, and it is generated by the class of ${\mathbb P}_H$. }

One should note that, as seen in Section \ref{more}, the results of Section 
\ref{hola} extend to the context of principal bundles. We shall conclude this paper 
by looking at our results in the context of Langlands duality in Section 
\ref{final}.

\section{Brauer group of some moduli spaces}\label{hola}

As in the introduction, we shall consider a compact connected Riemann surface 
$X$ of genus $g\,\geq\,3$ with a fixed base point $x_0\, \in\, X$, and denote by 
$K_X$ its canonical bundle.

\subsection{Brauer group of moduli spaces of connections}\label{se2.1}

A \textit{logarithmic connection} on $X$ singular over $x_0$ is a pair of the form $(E,\, D)$,
where $E$ is a holomorphic vector bundle on $X$ and
$$
D\,:\, E\, \longrightarrow\, E\otimes K_X\otimes {\mathcal O}_X(x_0)
$$
is a holomorphic differential operator of order one satisfying the Leibniz identity 
\begin{eqnarray}
D(fs) \,=\, f\cdot D(s) + s\otimes (df),
\end{eqnarray}
for all locally defined holomorphic functions $f$ on $X$ and
all locally defined holomorphic sections
$s$ of $E$. Note that the fiber $(K_X\otimes {\mathcal O}_X(x_0))_{x_0}$ is canonically
identified with $\mathbb C$ by sending any $c\, \in\, \mathbb C$ to the evaluation at $x_0$
of the locally defined section
$c\frac{dz}{z}$ of $K_X\otimes {\mathcal O}_X(x_0)$, where $z$ is any holomorphic function defined around
$x_0$ with $z(x_0)\,=\, 0$ and $(dz) (x_0)\,
\not=\, 0$. Moreover, the evaluation of $\frac{dz}{z}$ at $x_0$ does not depend on the choice
of the function $z$. Using this identification of $(K_X\otimes {\mathcal O}_X(x_0))_{x_0}$
with $\mathbb C$, for any logarithmic connection $D$ as above we have the linear
endomorphism of the fiber $E_{x_0}$ given by the composition
$$
E\, \stackrel{D}{\longrightarrow}\,E\otimes K_X\otimes {\mathcal O}_X(x_0)
\, \longrightarrow\, (E\otimes K_X\otimes {\mathcal O}_X(x_0))_{x_0}\,=\, E_{x_0}\, .
$$
This element of $\text{End}(E_{x_0})\,=\, E_{x_0}\otimes E^*_{x_0}$ is called the \textit{residue} of
$D$ (see \cite[p. 53]{De}), which we shall denote by ${\rm Res}(D,\,x_0)$.
Then, from \cite[pp. 16--17, Theorem 3]{Oh}, and \cite{De} one has \begin{equation}\label{e1}
\text{degree}(E)+ \text{trace}({\rm Res}(D,\,x_0)) \,=\, 0. 
\end{equation}
For notational convenience, we shall let $\mathbb{K}:=K_X\otimes {\mathcal O}_X(x_0)$.
\begin{definition}
A logarithmic connection $(E,\, D)$ as above is called \textit{semistable} (respectively,
\textit{stable}) if for every holomorphic subbundle $0\,\not=\, F\, \subsetneq\, E$ with
$D(F)\, \subset\, F\otimes \mathbb{K}$, the following inequality holds
$$
\frac{\text{degree}(F)}{\text{rank}(F)}\, \leq\, \frac{\text{degree}(E)}{\text{rank}(E)}\ \
{\rm \Big(respectively,~}\, \frac{\text{degree}(F)}{\text{rank}(F)}\, <\, \frac{\text{degree}(E)}{\text{rank}(E)}
{\rm \Big)}.
$$
 \end{definition}

As done in Section \ref{sec1}, fix two integers $r\, \geq\, 2$ and $\dd$, and if $g\,=\,3$, then assume
that $r\, > \, 2$. The holomorphic line bundle ${\mathcal O}_X(\dd x_0)$ on $X$ is equipped with
the logarithmic connection given by the de Rham differential $d$. This logarithmic connection
on ${\mathcal O}_X(\dd x_0)$ will be denoted by $D_0$. From \eqref{e1} it follows that the residue
of $D_0$ is $-\dd$.

In views of the notation introduced in Section \ref{sec1}, let $\overline{\mathcal M}_C$ denote the moduli space of logarithmic connections
$(E,\, D)$ on $X$, singular over $x_0$, satisfying the following three conditions:
\begin{enumerate}
\item[I.] $E$ is a holomorphic vector bundle on $X$ of rank $r$ with
$\bigwedge\nolimits^r E\,=\, {\mathcal O}_X(\dd x_0)$,

\item[II.] the logarithmic connection on $\bigwedge\nolimits^r E\,=\, {\mathcal O}_X(\dd x_0)$
induced by $D$ coincides with $D_0$ defined above, and

\item[III.] ${\rm Res}(D,\,x_0)\,=\, -\frac{\dd}{r}\text{Id}_{E_{x_0}}$.
\end{enumerate}
Note that from \eqref{e1}, the above condition on ${\rm Res}(D,\,x_0)$ implies that $(E,\, D)$ is automatically
semistable. Moreover, if $\dd$ is coprime to $r$, then $(E,\, D)$ is in fact stable. Since the residue
${\rm Res}(D,\,x_0)$ is a constant multiple of $\text{Id}_{E_{x_0}}$, the logarithmic connection on the
projective bundle $\mathbb{P}(E)$ induced by $D$ is actually regular at $x_0$. 

The above defined moduli space $\overline{\mathcal M}_C$ is a quasiprojective irreducible normal
variety, defined over $\mathbb C$, of dimension $2(r^2-1)(g-1)$. Let
\begin{equation}\label{e2}
{\mathcal M}_C\, \subset\,\overline{\mathcal M}_C
\end{equation}
be the Zariski open subset parametrizing the stable logarithmic connections. We note
that ${\mathcal M}_C$ is contained in the smooth locus of $\overline{\mathcal M}_C$ (in fact,
${\mathcal M}_C$ is the smooth locus of the space $\overline{\mathcal M}_C$).
We shall denote by $\text{Br}({\mathcal M}_C)$ the Brauer group of the
smooth variety ${\mathcal M}_C$ which, as mentioned in Section \ref{sec1}
is the Morita equivalence classes of Azumaya algebras over ${\mathcal M}_C$. The reader should refer to \cite{Gr1}, \cite{Gr2},
\cite{Gr3} \cite{Mi} for the definition as well as some properties of the Brauer group.

For any $(E,\, D)\, \in\, {\mathcal M}_C$, consider any $T\, \in\, H^0(X,\, \text{End}(E))$
which is flat with respect to the connection on $\text{End}(E)$ induced by $D$, or
equivalently, such that $D\circ T\,=\, (T\otimes\text{Id}_{\mathbb{K}})\circ D$. Then, for any $c\,
\in\, \mathbb C$, the kernel of $T-c\cdot \text{Id}_E\,\in\, H^0(X,\, \text{End}(E))$ is
preserved by $D$. Since $\text{kernel}(T-c\cdot \text{Id}_E)$ is either $E$ or $0$,
it follows that either $T\,=\, c\cdot \text{Id}_E$ or the endomorphism $T-c\cdot
\text{Id}_E$ is invertible. Now taking $c$ to be an eigenvalue of $T(x_0)$
it follows that $T\,=\, c\cdot \text{Id}_E$. Consequently, there is a universal projective bundle
\begin{equation}\label{pp}
\widetilde{\mathbb P}\,\longrightarrow\, X\times {\mathcal M}_C
\end{equation}
of relative dimension $r-1$ which is equipped with a relative holomorphic connection in the direction of $X$.

\begin{definition}
Given a point $x\, \in\, X$, let 
\begin{equation}\label{e3}
{\mathbb P}\,:=\, \widetilde{\mathbb P}\vert_{\{x\}\times {\mathcal M}_C}\,\longrightarrow\,{\mathcal M}_C
\end{equation}
be the projective bundle obtained by restricting $\widetilde{\mathbb P}$, and denote its class by \begin{equation}\label{e4}
\beta\,\in\, \text{Br}({\mathcal M}_C).
\end{equation}
\end{definition}

In order to study the Brauer group $\text{Br}({\mathcal M}_C)$ we shall first prove 
the following.

\begin{lemma}\label{lem1}
Let $Y$ be a simply connected smooth quasiprojective variety defined over
the complex numbers, $W$ an algebraic vector bundle on $Y$ and
$$
q\, :\, {\mathcal W}\, \longrightarrow\, Y
$$
a torsor on $Y$ for $W$. Then the pullback homomorphism
$$
q^*\, :\, {\rm Br}(Y)\, \longrightarrow\, {\rm Br}({\mathcal W})
$$
is an isomorphism.
\end{lemma}

\begin{proof}
Let $c\, \in\, H^1(Y,\, W)$ be the class of ${\mathcal W}$. Consider the extension of
${\mathcal O}_Y$ by $W$
\begin{equation}\label{xi}
0\, \longrightarrow\, W\, \longrightarrow\, \widehat{W} \, \stackrel{\xi}{\longrightarrow}
\, {\mathcal O}_Y \, \longrightarrow\, 0
\end{equation}
associated to the cohomology class $c$. We shall denote by $1_Y$ the image of the section of ${\mathcal O}_Y$
defined by the constant function $1$ on $Y$. Then, the inverse image
$\xi^{-1}(1_Y)\, \subset\, \widehat{W}$ under the projection $\xi$ in
\eqref{xi}, is a torsor on $Y$ for the vector bundle $W$.
This $W$--torsor is isomorphic to the $W$--torsor ${\mathcal W}$.

Let
$$\widehat{\eta}\, :\, P(\widehat{W}) \longrightarrow\, Y\ \ \text{ and }\ \ 
\eta\, :\, P(W)\longrightarrow\, Y$$
be the projective bundles on $Y$ parametrizing the lines in the fibers of
$\widehat{W}$ and $W$ respectively. Note that $P(W)\, \subset\, P(\widehat{W})$, and
$$
\mathcal{W}\,=\,\xi^{-1}(1_Y)\,=\, P(\widehat{W})\setminus P(W)\, ,
$$
by sending any element of $\xi^{-1}(1_Y)$ to the line in $\widehat W$
generated by it. 
Now from \cite[p.~365, Lemma~0.1]{Fo} and \cite[p.~367, Theorem~1.1]{Fo} we
conclude that the there is an exact sequence
\begin{equation}\label{xi2}
0\, \longrightarrow\, \text{Br}(P(\widehat{W})) \, \longrightarrow\,
\text{Br}(\mathcal{W})\, \longrightarrow\, H^1(P(W),\, {\mathbb Q}/{\mathbb Z})
\, \longrightarrow\, \ldots\, .
\end{equation}

Consider the long exact sequence of homotopy groups for the fiber bundle
$\eta$. The fibers of $P(W)$ are projective spaces and hence are
simply connected. Since $Y$ is also simply connected, from the homotopy exact
sequence for $\eta$ it follows that $P(W)$ is simply connected as well. Hence
$H_1(P(W),\, {\mathbb Z})\,=\, 0$, which implies that
$H^1(P(W),\, {\mathbb Q}/{\mathbb Z})\,=\, 0$ (universal coefficient theorem
for cohomology; see \cite[p.~195, Theorem~3.2]{Ha}).
Consequently, using \eqref{xi2} we conclude that
\begin{equation}\label{is.}
\text{Br}(P(\widehat{W})) \, =\,\text{Br}(\mathcal{W})
\end{equation}
with the isomorphism being induced by the inclusion of $\mathcal{W}$ in $P(\widehat{W})$.

The homomorphism $\widehat{\eta}^*\, :\, \text{Br}(Y)\, \longrightarrow\,\text{Br}
((P(\widehat{W}))$ induced by $\widehat{\eta}$
 is an isomorphism
\cite[p.~193, Theorem~2]{Ga}, and the lemma follows from \eqref{is.}.
\end{proof}

\begin{theorem}\label{theo1}
The Brauer group ${\rm Br}({\mathcal M}_C)$ is isomorphic to the cyclic
group ${\mathbb Z}/\tau\mathbb Z$, where $\tau\,=\, {\rm g.c.d.}(r,\dd)$, and it is generated by the element $\beta$ in \eqref{e4}.
\end{theorem}

\begin{proof}
Let ${\mathcal N}$ denote the moduli space of stable vector bundles $V$ on $X$ of rank $r$ such that
$\bigwedge\nolimits^r V\,=\, {\mathcal O}_X(\dd x_0)$, which is a smooth quasiprojective irreducible complex
variety of dimension $(r^2-1)(g-1)$.
Moreover, let ${\mathcal U}\, \subset\, {\mathcal M}_C$ be the locus of all $(E,\, D)$ such that the underlying
holomorphic vector bundle $E$ is stable. Considering 
\begin{equation}\label{e5}
p\, :\, {\mathcal U}\,\longrightarrow\, \mathcal N
\end{equation}
the forgetful morphism that sends any $(E,\, D)$ to $E$, from the openness of the stability condition
(see \cite[p. 635, Theorem 2.8(B)]{Ma}) it follows that $\mathcal U$ is a Zariski open subset of ${\mathcal M}_C$.
Note that $p$ is surjective because from \cite{NS} one has that
any $V\, \in\, \mathcal N$ admits a unique logarithmic connection $D_V$ singular at $x_0$ such that
\begin{itemize}
\item[(a)] the residue of $D_V$ at $x_0$ is $-\frac{\dd}{r}\text{Id}_{V_{x_0}}$, and

\item[(b)] the monodromy of $D_V$ lies in $\text{SU}(r)$.
\end{itemize}
Moreover, a pair $(V,\, D_V)$ as above lies in ${\mathcal U}$. In fact, if $D'$ is a logarithmic connection on
$V$ singular at $x_0$ such that ${\rm Res}(D',\,x_0)\,=\, -\frac{\dd}{r}\text{Id}_{V_{x_0}}$,
and the logarithmic connection on $\bigwedge\nolimits^r V\,=\, {\mathcal O}_X(\dd x_0)$
induced by $D'$ coincides with $D_0$, then
clearly $(V,\, D')\, \in\, {\mathcal U}$. The space of all logarithmic connections $D'$ on $V$ satisfying
the conditions (a) and (b) is an affine space for the vector space
$H^0(X,\, \text{ad}(V)\otimes K_X)$, where $\text{ad}(V)\subset\, \text{End}(V)$ is the
subbundle of co-rank one defined by the sheaf of endomorphisms of trace zero.
Furthermore, $H^0(X,\, \text{ad}(V)\otimes K_X)$ is the fiber of the cotangent bundle $\Omega^1_{\mathcal N}$ over
the point $V\, \in\, \mathcal N$. Therefore, the morphism $p$ in \eqref{e5} makes $\mathcal U$ a torsor
over $\mathcal N$ for the vector bundle $\Omega^1_{\mathcal N}$.

From \cite[p. 301, Lemma 3.1]{BM1} and \cite[Lemma 3.1]{BM2} the complement ${\mathcal M}_C\setminus {\mathcal U}\, \subset\, {\mathcal M}_C$ is of codimension at least
two and thus the inclusion map
$\iota\, :\, {\mathcal U}\, \hookrightarrow\, {\mathcal M}_C$ produces an isomorphism of Brauer groups 
\begin{equation}\label{e6}
\iota^*\,:\, \text{Br}({\mathcal M}_C)\,\stackrel{\sim}{\longrightarrow}\, \text{Br}({\mathcal U})\, ;
\end{equation}
this follows from ``Cohomological purity'' \cite[p.~241, Theorem~VI.5.1]{Mi} (it also
follows from \cite[p.~292--293]{Gr2}).
Since $p$ in \eqref{e5} is a torsor on ${\mathcal U}$ for
a vector bundle, and ${\mathcal U}$ is simply connected
\cite[p.~266, Proposition~1.2(b)]{BBGN}, from Lemma \ref{lem1}
it follows that the map $p$ induces an isomorphism
$$
p^*\,:\, \text{Br}({\mathcal N})\,\stackrel{\sim}{\longrightarrow}\,
\text{Br}({\mathcal U})\, .
$$
Combining this with \eqref{e6} we get an isomorphism
\begin{equation}\label{e7}
(\iota^*)^{-1}\circ p^*\,:\,
\text{Br}({\mathcal N})\,\stackrel{\sim}{\longrightarrow}\, \text{Br}({\mathcal M}_C)\, .
\end{equation}

We know that $\text{Br}({\mathcal N})$ is cyclic of order $\tau\,=\,
{\rm g.c.d.}(r,\dd)$ \cite[p. 267, Theorem 1.8]{BBGN}. Therefore, from \eqref{e7} we conclude that
$\text{Br}({\mathcal M}_C)$ is isomorphic to ${\mathbb Z}/\tau\mathbb Z$.

Finally, in order to find a generator of $\text{Br}({\mathcal M}_C)$, let
$
\widehat{\mathbb P}\,\longrightarrow\, X\times{\mathcal N}
$
be the universal projective bundle and
$$
{\mathbb P}'\,:=\, \widehat{\mathbb P}\vert_{\{x\}\times {\mathcal N}}\,\longrightarrow\,{\mathcal N}
$$
be the projective bundle obtained by restricting $\widehat{\mathbb P}$, where $x$
is the point of $X$ in \eqref{e3}. The Brauer group
$\text{Br}({\mathcal N})$ is generated by the class of ${\mathbb P}'$
\cite[p. 267, Theorem 1.8]{BBGN}. The pulled back projective bundle
$(\text{Id}_X\times p)^*\widehat{\mathbb P}$ is identified with the restriction
$\widetilde{\mathbb P}\vert_{X\times {\mathcal U}}$, where
$\widetilde{\mathbb P}$ is the projective bundle in \eqref{pp}. This implies
that $p^*{\mathbb P}'$ is identified with the restriction
${\mathbb P}\vert_{\mathcal U}$. Since the class of ${\mathbb P}'$ generates $\text{Br}({\mathcal
N})$, from the isomorphism $(\iota^*)^{-1}\circ p^*$ in \eqref{e7} it follows that the class of ${\mathbb P}$
generates $\text{Br}({\mathcal M}_C)$.
\end{proof}

\subsection{Analytic Brauer group of representation spaces}

Consider now the free group $\Gamma$ generated by $2g$ elements $\{a_i\, ,b_i\}_{i=1}^g$, and let
\begin{equation}\label{ga}
\gamma\, :=\, \prod_{i=1}^g [a_i\, , b_i]\,=\, \prod_{i=1}^g (a_ib_ia^{-1}_ib^{-1}_i)\, \in\,
\Gamma
\end{equation}
be the product of the commutators. Then, one may consider the space of all homomorphisms $\text{Hom}(\Gamma ,\, 
\text{SL}(r,{\mathbb C}))$ from the group $\Gamma$ to $\text{SL}(r,{\mathbb C})$. Let
$$
\text{Hom}^{\dd}(\Gamma ,\, \text{SL}(r,{\mathbb C}))\, \subset\, \text{Hom}(\Gamma ,\, 
\text{SL}(r,{\mathbb C}))
$$
be all such homomorphisms $\rho$ satisfying the condition that
\[\rho(\gamma)\,=\,
\exp(2\pi\sqrt{-1}\dd/r)\cdot I_{r\times r}\, ,\]
 where $I_{r\times r}$ is the $r\times r$
identity matrix. The conjugation action of $\text{SL}(r,{\mathbb C})$ on itself produces an
action of $\text{SL}(r,{\mathbb C})$ on the variety
$\text{Hom}(\Gamma ,\, \text{SL}(r,{\mathbb
C}))$. Moreover, this action of $\text{SL}(r,{\mathbb C})$
on $\text{Hom}(\Gamma ,\, \text{SL}(r,{\mathbb C}))$ preserves the above subvariety
$\text{Hom}^{\dd}(\Gamma ,\, \text{SL}(r,{\mathbb C}))$. We shall denote by $\overline{\mathcal R}$ the geometric invariant
theoretic quotient
$$
\overline{\mathcal R}\, :=\, 
\text{Hom}^{\dd}(\Gamma\, , \text{SL}(r,{\mathbb C}))/\!\!/\text{SL}(r,{\mathbb C})\, ,
$$
which parametrizes all the closed orbits of $\text{SL}(r,{\mathbb C})$
in $\text{Hom}^{\dd}(\Gamma ,\, \text{SL}(r,{\mathbb C}))$.

The moduli space $\overline{\mathcal M}_C$ defined in Section \ref{se2.1} is biholomorphic to
$\overline{\mathcal R}$. After fixing a presentation of $\pi_1(X\setminus\{x_0\},\, x)$, we
have a map
$\overline{\mathcal M}_C\,\longrightarrow\, \overline{\mathcal R}$ that sends a flat
connection to its monodromy representation, and which is a biholomorphism. Indeed, it is
the inverse of the map that associates a flat bundle on $X\setminus\{x_0\}$ of rank $r$ to a
representation of \linebreak $\pi_1(X\setminus\{x_0\},\, x)$. Note that although
both $\overline{\mathcal M}_C$ and $\overline{\mathcal R}$ have natural algebraic
structures, the above biholomorphism between them is not an algebraic map.

A homomorphism $\rho\, :\, \Gamma\,\longrightarrow\, \text{SL}(r,{\mathbb C})$ is called
\textit{reducible} if $\rho(\Gamma)$ is contained in some proper parabolic subgroup of
$\text{SL}(r,{\mathbb C})$, equivalently, if $\rho(\Gamma)$ preserves some
proper nonzero subspace of ${\mathbb C}^r$ under the standard action of
$\text{SL}(r,{\mathbb C})$. If $\rho$ is not reducible, then it is called
\textit{irreducible}.

We shall denote by 
$$
{\mathcal R}\, \subset\,\overline{\mathcal R}
$$
the locus of irreducible homomorphisms in $\overline{\mathcal R}$. It is easy to
see that ${\mathcal R}$ is contained in the smooth locus of $\overline{\mathcal R}$ (in fact, $\mathcal R$ coincides with the smooth locus of $\overline{\mathcal R}$).
{}From the definitions of ${\mathcal M}_C$ and $\mathcal R$ it follows immediately that 
the above biholomorphism $\overline{\mathcal M}_C\,
\stackrel{\sim}{\longrightarrow}\, \overline{\mathcal R}$ produces a biholomorphism
\begin{equation}\label{e9}
{\mathcal M}_C\,\stackrel{\sim}{\longrightarrow}\, {\mathcal R}\, .
\end{equation}

Let ${\mathcal H}\, \subset\, \text{Hom}^{\dd}(\Gamma\, , \text{SL}(r,{\mathbb C}))$ be the
inverse image of ${\mathcal R}$; in other words, ${\mathcal H}$ is the locus of all elements of
$\text{Hom}^{\dd}(\Gamma\, , \text{SL}(r,{\mathbb C}))$ that are
irreducible homomorphisms. The quotient map
$$
{\mathcal H}\,\longrightarrow\, {\mathcal H}/\!\!/\text{SL}(r,{\mathbb C})\,=\,
{\mathcal R}
$$
makes ${\mathcal H}$ an algebraic principal $\text{PSL}(r,{\mathbb C})$--bundle over
${\mathcal R}$. We shall denote by
\begin{equation}\label{e10}
\mathbb{P}_R\, :=\, {\mathcal H}\times^{\text{PSL}(r,{\mathbb C})} \mathbb{C}{\mathbb P}^{r-1}
\,\longrightarrow\, {\mathcal R}
\end{equation}
the fiber bundle associated to the principal $\text{PSL}(r,{\mathbb C})$--bundle
${\mathcal H}\,\longrightarrow\,{\mathcal R}$ for the standard action of $\text{PSL}(r,{\mathbb C})$
on $\mathbb{C}{\mathbb P}^{r-1}$.

The analytic Brauer group of ${\mathcal R}$ is
defined to be the equivalence classes of holomorphic
principal $\text{PGL}$--bundles on ${\mathcal R}$ where two principal $\text{PGL}$--bundles
$P$ and $Q$ are equivalent if there are holomorphic vector bundles $V$ and $W$ on
${\mathcal R}$ such that the two principal $\text{PGL}$--bundles $P\otimes {\mathbb P}(V)$ and
$Q\otimes {\mathbb P}(W)$ are isomorphic. Moreover, the analytic cohomological Brauer
group of ${\mathcal R}$ is the torsion part of $H^2({\mathcal R},\,
{\mathcal O}^*_{\mathcal R})$ (see \cite{Sc}). Let ${\rm Br}_{\rm an}({\mathcal R})$
(respectively, ${\rm Br}'_{\rm an}({\mathcal R})$) denote the
analytic Brauer group (respectively, analytic cohomological Brauer group) of
${\mathcal R}$. Then, from \cite[p.~878]{Sc} one has that
$${\rm Br}_{\rm an}({\mathcal R})\, \subset\,
{\rm Br}'_{\rm an}({\mathcal R})\, .$$

\begin{theorem}\label{thm1}
The analytic cohomological Brauer group ${\rm Br}'_{\rm an}({\mathcal R})$ is
isomorphic to a quotient of the cyclic group ${\mathbb Z}/\tau
\mathbb Z$, where $\tau\,=\, {\rm g.c.d.}(r,\dd)$. Moreover, the group
${\rm Br}_{\rm an}({\mathcal R})$ is generated by
the class of the projective bundle $\mathbb{P}_R$ in \eqref{e10}.
\end{theorem}

\begin{proof}
{}From the biholomorphism in \eqref{e9}, the group ${\rm Br}'_{\rm an}({\mathcal R})$
coincides with the
analytic Brauer group ${\rm Br}'_{\rm an}({\mathcal M}_C)$ of the stable moduli space
${\mathcal M}_C$. Moreover, the forgetful map ${\rm Br}'({\mathcal M}_C)\,
\longrightarrow\, {\rm Br}'_{\rm an}({\mathcal M}_C)$ is surjective
\cite[p.~879, Proposition~1.3]{Sc}. Then, since ${\rm Br}'({\mathcal M}_C)\,=\,
{\rm Br}({\mathcal M}_C)$, we conclude that ${\rm Br}'_{\rm an}({\mathcal R})$
is a quotient of ${\rm Br}'({\mathcal M}_C)$. Therefore, from the first part of
Theorem \ref{theo1} it follows that
${\rm Br}'_{\rm an}({\mathcal R})$ is a quotient of the cyclic
group ${\mathbb Z}/\tau\mathbb Z$.

The isomorphism in \eqref{e9} takes the projective bundle $\mathbb{P}_R$ constructed
in \eqref{e10} holomorphically to the projective bundle ${\mathbb P}$ of \eqref{e3}.
Therefore, from the second part of Theorem \ref{theo1}
it follows that ${\rm Br}'_{\rm an}({\mathcal R})$ is generated by the class of
$\mathbb{P}_R$.
\end{proof}
Note that whilst the natural homomorphism ${\rm Br}_{\rm an}({\mathcal R})\,\longrightarrow\,
{\rm Br}'_{\rm an}({\mathcal R})$ is injective \cite[p.~878]{Sc}, Theorem \ref{thm1}
implies that this homomorphism is surjective and thus the following corollary
is proved.

\begin{corollary}\label{cor1}
The analytic Brauer group ${\rm Br}_{\rm an}({\mathcal R})$ coincides with
${\rm Br}'_{\rm an}({\mathcal R})$.
\end{corollary}

Regarding the above Corollary \ref{cor1} it should be clarified that the analog of Gabber's
theorem, which would say that the Brauer group coincides with the cohomological Brauer group, is not
available in the analytic category.

\subsection{Brauer group of moduli spaces of Higgs bundles}\label{sec2.3}

We shall now consider Higgs bundles on a compact Riemann surface. 
As in Section \ref{se2.1}, consider a compact connected Riemann surface $X$ of genus $g \geq 3$, and $x_0\,\in\, X$ a base point. Let ${\mathcal M}_H$ denote the moduli space of stable
Higgs bundles on $X$ of the form $(E,\, \Phi)$, where $E$ is a holomorphic vector bundle
of rank $r$ with $\bigwedge\nolimits^r E\,=\, {\mathcal O}_X(\dd x_0)$, and $\Phi$ is a
Higgs field on $X$ with $\text{Tr}(\Phi)\,=\, 0$. The moduli space ${\mathcal M}_H$ is an
irreducible smooth complex quasiprojective variety of dimension $2(r^2-1)(g-1)$.

Consider the moduli space $\mathcal N$ from \eqref{e5}, for which the total space of the algebraic
cotangent bundle $T^*\mathcal N$ is embedded in in ${\mathcal M}_H$ as a Zariski open
subset. The codimension of the complement ${\mathcal M}_H\setminus T^*\mathcal N$
is at-least two \cite{Hi}. Therefore, by purity of cohomology, and Lemma
\ref{lem1}, one has that
$$
\text{Br}(\mathcal N)\,=\, \text{Br}(T^*\mathcal N)\,=\, \text{Br}({\mathcal M}_H)\, ;
$$
as before we use that $\mathcal N$ is simply connected.

Hence we have the following:

\begin{proposition}\label{mas}
The Brauer group ${\rm Br}({\mathcal M}_H)$ is the cyclic group of order
${\rm g.c.d.}(r,\dd)$.
Furthermore, ${\rm Br}({\mathcal M}_H)$ is generated by the class of the projective
bundle on ${\mathcal M}_H$ obtained by restricting to
$\{x\}\times {\mathcal M}_H$ the universal projective bundle on $X\times {\mathcal M}_H$.
\end{proposition}

\section{Generalizations to principal bundles} \label{more}

Let $G$ be a semisimple simply connected affine algebraic group defined over 
$\mathbb C$. The topological types of principal $G$--bundles on $X$ are 
parametrized by $\pi_1(G)$, which is a finite abelian group. Let 
$\overline{\mathcal M}_C(G)$ denote the moduli space of pairs of the form $(E_G,\, 
D)$, where $E_G$ is a topologically trivial holomorphic principal $G$--bundle on
$X$, and $D$ is a holomorphic connection on $E_G$. Following the notation from the
previous sections, let
$$
{\mathcal M}_C(G)\, \subset\, \overline{\mathcal M}_C(G)
$$
be the smooth locus of $\overline{\mathcal M}_C(G)$.

The center of $G$ will be denoted by $Z(G)$. A stable principal $G$--bundle
is called {\it regularly stable} if its automorphism group coincides with $Z(G)$.
We shall denote by ${\mathcal N}_G$ the moduli space of regularly stable principal
$G$--bundles on $X$ that are topologically trivial.
Recall from \cite{BH} that the Brauer group $\text{Br}({\mathcal N}_G)$ can be identified with
the dual group\linebreak $Z(G)^\vee\,=\, \text{Hom}(Z(G),\, {\mathbb C}^*)$, and $\text{Br}
({\mathcal N}_G)$ is generated by the tautological
$Z(G)$--gerbe on ${\mathcal N}_G$ defined by the moduli stack of regularly
stable topologically trivial principal $G$--bundles on $X$. Note that given any homomorphism
$Z(G)^\vee\, \longrightarrow\, {\mathbb G}_m$, using extension of structure group
the above $Z(G)$--gerbe produces a ${\mathbb G}_m$--gerbe on ${\mathcal N}_G$.

\begin{proposition}\label{prop1c}
The Brauer group ${\rm Br}({\mathcal M}_C(G))$ is isomorphic to
the dual group $Z(G)^\vee$ and is generated by
the tautological $Z(G)$--gerbe on ${\mathcal M}_C(G)$.
\end{proposition}

\begin{proof}
A straight-forward generalization of the proof of Theorem \ref{theo1} proves the proposition. We note that ${\mathcal N}_G$ is simply connected \cite[p.~416, Theorem~1.1]{BLR},
hence Lemma \ref{lem1} is applicable.
\end{proof}

Similarly, the (analytic) Brauer group computations in Theorem \ref{thm1} and
Section \ref{sec2.3} extend to $G$.

\section{Langlands duality and Brauer groups} \label{final}

As previously, suppose that $G$ is simply connected and let ${}^LG$ denote the
Langlands dual group. 
There is a canonical isomorphism $\pi_1( {}^LG ) \,\cong\, Z(G)^\vee$, which can be seen 
from the root-theoretic construction of the Langlands dual. We shall denote by $\mathcal{M}_H(G)$ and
$\mathcal{M}_H( {}^LG )$ the moduli spaces of Higgs bundles for the groups $G$ 
and ${}^LG$ respectively.

 It is known that the connected components of $\mathcal{M}_H( {}^LG)$ 
correspond to $\pi_1( {}^LG )$, by taking the topological class of the underlying 
principal bundle. Recall that the moduli spaces $\mathcal{M}_H(G)$ and $\mathcal{M}_H( 
{}^L G)$ satisfy SYZ mirror symmetry, that is, they are dual special Lagrangian torus 
fibrations over a common base \cite{dopa}. Under this duality, the choice of a 
connected component of $\mathcal{M}_H( {}^L G)$ corresponds to fixing a 
$\mathbb{C}^*$--gerbe on $\mathcal{M}_H( G )$. Namely, given a class in $\pi_1( {}^LG ) 
\,\cong\, Z(G)^\vee$, we obtain from the universal $G/Z(G)$--bundle on $\mathcal{M}_H(G)$ 
the corresponding $\mathbb{C}^*$--gerbe. Our computations show that every 
class in ${\rm Br}( \mathcal{M}_H(G))$ is accounted for by this correspondence.
 
\section*{Acknowledgements}

We are very grateful to T. J. Ford for a helpful correspondence. We thank the 
Institute for Mathematical Sciences in the National University of Singapore for 
hospitality while this work was being completed. IB acknowledges support
of a J. C. Bose Fellowship. This research was completed with 
support from U.S. National Science Foundation grants DMS 1107452, 1107263, 1107367 
ÒsRNMS: GEometric structures And Representation varietiesÓ (the GEAR Network). The 
work of LPS is partially supported by NSF DMS-1509693.


\end{document}